\documentclass[letterpaper,11pt]{amsart}

\usepackage[margin=1.2in]{geometry}
\usepackage{amsmath,amsthm,amssymb}
\usepackage{xspace,xcolor}
\usepackage[breaklinks,colorlinks,citecolor=teal,linkcolor=teal,urlcolor=teal,pagebackref,hyperindex]{hyperref}
\usepackage[alphabetic]{amsrefs}
\usepackage[all]{xy}
\usepackage{color}

\theoremstyle{plain}
\newtheorem{thm}{Theorem}[section]

\newtheorem{lem}[thm]{Lemma}
\newtheorem{prop}[thm]{Proposition}
\newtheorem{cor}[thm]{Corollary}

\theoremstyle{definition}

\newtheorem{eg}[thm]{Example}

\theoremstyle{remark}
\newtheorem{rmk}[thm]{Remark}

\def\Z{{\mathbf Z}}
\def\Q{{\mathbf Q}}

\def\C{{\mathbf C}}

\def\A{{\mathbf A}}
\def\P{{\mathbf P}}

\def\cA{\mathcal{A}}

\def\cD{\mathcal{D}}
\def\cE{\mathcal{E}}
\def\cF{\mathcal{F}}

\def\cH{\mathcal{H}}
\def\cI{\mathcal{I}}

\def\cL{\mathcal{L}}
\def\cM{\mathcal{M}}

\def\cO{\mathcal{O}}

\def\RHom{{\mathcal R}{\mathcal Hom}}

\def\.{\cdot}
\def\^{\widehat}

\def\({\left(}
\def\){\right)}

\newcommand{\llbracket}{[\negthinspace[}
\newcommand{\rrbracket}{]\negthinspace]}

\newcommand{\llparenthesis}{(\negthinspace(}
\newcommand{\rrparenthesis}{)\negthinspace)}

\renewcommand{\and}{ \ \ \text{ and } \ \ }

\begin{document}

\author{Bradley Dirks}
\author{Mircea Musta\c{t}\u{a}}

\address{Department of Mathematics, University of Michigan, 530 Church Street, Ann Arbor, MI 48109, USA}

\email{bdirks@umich.edu}
\email{mmustata@umich.edu}

\thanks{The authors were partially supported by NSF grant DMS-1701622.}


\begin{abstract}
Given a reduced effective divisor $D$ on a smooth variety $X$, we describe the generating function for the classes of the Hodge ideals
of $D$ in the Grothendieck group of coherent sheaves on $X$ in terms of the motivic Chern class of the complement of the support of $D$.
As an application, we compute the generating function for the Hilbert series of Hodge ideals of a hyperplane arrangement in terms of the 
Poincar\'{e} polynomial of the arrangement.

\end{abstract}

\title{The Hilbert series of Hodge ideals of hyperplane arrangements}

\maketitle

\section{Introduction} 

Let $X$ be a smooth complex algebraic variety and $D$ 
a reduced effective divisor on $X$. By making use of Saito's theory of mixed Hodge modules \cite{Saito-MHM}, one can attach to $D$
a sequence of ideals, the \emph{Hodge ideals} $I_p(D)$ for $p\geq 0$, that have been systematically studied in \cite{MP}. The first ideal
$I_0(D)$ is the multiplier ideal $\cI\big((1-\epsilon)D\big)$, for $0<\epsilon\ll 1$ and the higher ideals can be viewed as similar but more
refined measures of the singularities of $D$. In this note we focus on the classes of the Hodge ideals in the
Grothendieck group $K_0(X)$ of coherent sheaves on $X$, as encoded in the generating function 
$\sum_{p\geq 0}\big[I_p(D)\big]y^p\in K_0(X)\llbracket y\rrbracket$. 

Our main observation is that this generating function can be described in terms of the motivic Chern class of the inclusion $j\colon U\hookrightarrow X$,
where $U$ is the complement of the support of $D$. Recall that Brasselet, Sch\"{u}rmann, and Yokura introduced in \cite{BSY} this motivic Chern class,
which is a group homomorphism
$$mC_y\colon K_0({\rm Var}/X)\to K_0(X)[y],$$
where $K_0({\rm Var}/X)$ is the Grothendieck group of varieties over $X$.
The motivic Chern class is uniquely characterized by the fact that it commutes with push-forward with respect to proper morphisms
and that it satisfies a normalization condition, given by its value on the identity map of a smooth variety. 
The existence of the transformation can be easily deduced from Bittner's presentation \cite{Bittner}
of $K_0({\rm Var}/X)$ via blow-up relations. For us, it is important that there is 
an explicit description of the transformation via mixed Hodge modules: given a variety $f\colon Y\to X$ over $X$, the motivic Chern class
of $f$ can be described via the (de Rham complex associated to the) object $f_!\Q^H_Y$ in the derived category of mixed Hodge modules on $X$.
On the other hand, the Hodge ideals of $D$ describe the Hodge filtration on $j_*\Q^H_U[n]$, where $n=\dim(X)$, and duality allows us to relate
the graded objects associated to the de Rham complexes of $j_!\Q^H_U$ and $j_*\Q^H_U[n]$. This allows us to write down $\sum_{p\geq 0}\big[I_p(D)\big]y^p$ in terms of
$mC_y\big([U\hookrightarrow Y])$. For the precise statement, see Theorem~\ref{thm_main1}.

When we are in the presence of a group action, it is convenient to take this into account. There is an equivariant version
of the motivic Chern class and this has been computed in many interesting cases
(see for example \cite{AMSS} and \cite{FRW}). If $G$ is a linear algebraic group acting on $X$ and the divisor $D$ is
$G$-invariant, then the ideals $I_p(D)$ are preserved by the $G$-action and we may consider their classes in the Grothendieck group $K_0^G(X)$
of $G$-equivariant
coherent sheaves on $X$. Our formula for the generating function of the Hodge ideals holds more generally in this equivariant setting, using the
equivariant motivic Chern class. 

Our main application in this paper is to the case when $X=V$ is a complex vector space and $D=D_{\cA}$ is the divisor corresponding to an arrangement $\cA$
of linear hyperplanes in $X$. We consider the standard action of $T=\C^*$ on $X$, so that describing $I_p(D_{\cA})$ in $K_0^T(V)$ is equivalent to describing
the Hilbert series $H_{I_p(D_{\cA})}(t)$ of $I_p(D_{\cA})$. Since the equivariant motivic Chern class is easy to compute in this setting, we obtain the following description of the 
generating function of the Hilbert series $H_{I_p(D_{\cA})}$ in terms of the Poincar\'{e} polynomial $\pi(\cA,x)$ of the arrangement.

\begin{thm}\label{thm_main2}
If $\cA$ is a central hyperplane arrangement of $d$ hyperplanes in an $n$-dimensional complex vector space $V$
and if $D_{\cA}=\sum_{H\in\cA}H$, then
$$\sum_{p\geq 0}H_{I_p(D_{\cA})}(t)y^p=\frac{t^{d}}{(1-t)^n(1-t^dy)}\cdot\pi\big(\cA,(1-t)/t(1-t^{d-1}y)\big).$$
\end{thm}

By letting $y=0$ in the above theorem and recalling the identification of $I_0(D)$ with a multiplier ideal,
we obtain the following

\begin{cor}\label{cor_mult_id}\label{cor_main2}
If $\cA$ is a central hyperplane arrangement of $d$ hyperplanes in an $n$-dimensional complex vector space $V$
and if $D_{\cA}=\sum_{H\in\cA}H$, then the Hilbert series of the multiplier ideal $I=\cI\big((1-\epsilon)D_{\cA}\big)$, with $0<\epsilon\ll 1$, is given by 
$$H_I(t)=\frac{t^{d}}{(1-t)^n}\cdot\pi(\cA,t^{-1}-1).$$
\end{cor}

We note that all multiplier ideals of a hyperplane arrangement $\cA$ admit an explicit geometric description as intersections of suitable powers 
of the ideals defining various intersections of hyperplanes in $\cA$ (see \cite{Mustata} and \cite{Teitler}). However, it is not clear to us how to recover the formula in Corollary~\ref{cor_mult_id} from this description.

It is an interesting question whether there is a generalization of the above results to rational coefficients. Recall that one can define 
more generally Hodge ideals $I_p(\alpha D)$ for positive rational numbers $\alpha$ (see \cite{MP0})
and it is thus natural to investigate generating functions of the form $\sum_{p\geq 0}\big[I_p(\alpha D)\big]y^p$. 
The methods in this paper easily extend (at least, in the case when $D$ is defined by a global equation) to describe
this generating function in terms of the equivariant motivic Chern class of a suitable \'{e}tale cyclic cover of the complement of the support of $D$.
However, in the case of a hyperplane arrangement, computing the motivic Chern class of this cover seems to require a new idea.

The paper is structured as follows. In Section 2 we set up some notation and review some basic facts about equivariant Grothendieck groups.
In Section 3 we discuss the motivic Chern class, following \cite{BSY}, especially its description via mixed Hodge module theory. We also discuss the equivariant case, which makes use of Achar's construction of an equivariant derived category of mixed Hodge modules \cite{Achar}.
In Section 4 we prove the formula relating the generating function for the classes of the Hodge ideals in the Grothendieck group and the motivic Chern class. 
We then turn to the setting of hyperplane arrangements. After reviewing the description of the $\C^*$-equivariant Grothendieck group on the affine space in 
Section 5, we compute in Section 6 the equivariant motivic Chern class of the complement of the arrangement and deduce the results in
Theorem~\ref{thm_main2} and Corollary~\ref{cor_main2}. In the last section we discuss an easy example, that of a hyperplane arrangement consisting
of coordinate hyperplanes, in which we give a direct computation of the Hilbert series of the Hodge ideals to recover the formula obtained via
Theorem~\ref{thm_main2}.

\subsection{Acknowledgment} 
We would like to thank J\"{o}rg Sch\"{u}rmann for his detailed comments and suggestions regarding the definition of the motivic Chern class
in the equivariant setting. The second author is also grateful to Mihnea Popa for many discussions. 

\section{Equivariant Grothendieck groups}

In this section we set up some notation and review  basic facts about equivariant Grothendieck groups of algebraic varieties first,
and then about Grothendieck groups of equivariant sheaves. We only consider 
\emph{complex algebraic varieties}, by which we mean reduced, separated schemes of finite type over $\C$,  not necessarily irreducible.
Let $G$ be a linear (that is, affine) algebraic group. A $G$-variety is a variety $X$ with an algebraic action of $G$ and a morphism of
$G$-varieties is assumed to be an equivariant morphism. Given a $G$-variety $X$, a \emph{$G$-variety over $X$} is a $G$-variety $Y$, with a morphism
of $G$-varieties $f\colon Y\to X$.

Given a $G$-variety $X$,
the \emph{Grothendieck group of $G$-varieties over $X$} is the quotient $K^G_0({\rm Var}/X)$ 
of the free Abelian group on isomorphism classes of $G$-varieties over $X$, modulo the \emph{cut-and-paste} relations
$$[Y\to X]=[Z\to X]+[U\to X],$$
for a $G$-variety $Y$ over $X$, a $G$-invariant closed subvariety $Z$ of $Y$, and $U=Y\smallsetminus Z$. 
In fact, this has a ring structure induced by the fiber product of varieties over $X$, with the induced $G$-action.
If $f\colon X'\to X$ is a morphism of $G$-varieties, then every $G$-variety over $X'$ has an induced structure of $G$-variety over $X$ and we get in this way
a morphism of Abelian groups
$$f_!\colon K^G_0({\rm Var}/X')\to K^G_0({\rm Var}/X).$$
If $G=\{1\}$, then we simply write $K_0({\rm Var}/X)$ for $K_0^G({\rm Var}/X)$.

We next turn to the definitions concerning the Grothendieck groups of equivariant sheaves on $X$. For more details about 
equivariant $K$-theory, see \cite{Thomason}. Given a $G$-variety $X$, an equivariant coherent sheaf on $X$ is a coherent sheaf 
$\cF$, together with an isomorphism $\theta^*(\cF)\simeq p_2^*(\cF)$, where $\theta\colon G\times X\to X$ is the action and $p_2\colon G\times X\to X$
is the projection (moreover, this isomorphism is supposed to satisfy an obvious cocycle condition). In particular, we can consider equivariant
locally free sheaves on $X$. We have an obvious notion of equivariant morphism of coherent sheaves.

The \emph{Grothendieck group of equivariant coherent sheaves} on $X$ is the quotient $K^G_0(X)$ of the free Abelian group on isomorphism 
classes of equivariant coherent sheaves, modulo the relations
$$[\cF']-[\cF] +[\cF'']=0,$$
where
$$0\to \cF'\to \cF\to\cF''\to 0$$
is a short exact sequence of equivariant coherent sheaves on $X$. If we replace ``equivariant coherent sheaves" by ``equivariant
locally free sheaves", we obtain the \emph{Grothendieck ring} $K_G^0(X)$ of equivariant locally free sheaves on $X$. This is indeed a commutative ring, with the product induced by tensor product and the multiplicative 
identity given by the structure sheaf. If $G=\{1\}$, then we simply write $K_0(X)$ and $K^0(X)$, respectively, for $K_0^G(X)$ and $K^0_G(X)$.

 Note also that the tensor product makes $K^G_0(X)$ a module over $K^0_G(X)$. By mapping the class of an equivariant locally free sheaf
 to itself, viewed as an equivariant coherent sheaf, we get a morphism of $K^0_G(X)$-modules
 $$K^0_G(X)\to K_0^G(X).$$
 If $f\colon Y\to X$ is a morphism of $G$-varieties, then by pulling-back equivariant locally free sheaves we obtain a ring homomorphism
 $$f^*\colon K_G^0(X)\to K^0_G(Y).$$
 If $f$ is proper, 
 then we also have a group homomorphism
 $$K_0^G(Y)\to K^G_0(X),\quad[\cF]\to\sum_{i\geq 0}(-1)^i\big[R^if_*(\cF)\big].$$

 For a $G$-variety $X$, we will also consider the Abelian group $K^G_0\llparenthesis y\rrparenthesis$ of Laurent power series
 with coefficients in $K^G_0(X)$. This has a natural module structure over the ring $K^0_G\llparenthesis y\rrparenthesis$. We can similarly consider
 $K^G_0\llparenthesis y^{-1}\rrparenthesis$ and $K_G^0\llparenthesis y^{-1}\rrparenthesis$.

 Let $X$ be a fixed $G$-variety.
For an equivariant locally free sheaf $\cE$ on $X$, of rank $r$, we put
$$\lambda_y(\cE):=\sum_{i=0}^r[\wedge^i\cE]y^i\in K_G^0(X)[y].$$
It is easy to check (and well-known) that if we have a short exact sequence
of equivariant locally free sheaves
\begin{equation}\label{eq_exact_seq}
0\to\cE'\to\cE\to\cE''\to 0,
\end{equation}
then
\begin{equation}\label{eq_add_lambda}
\lambda_y(\cE)=\lambda_y(\cE_1)\cdot\lambda_y(\cE_2).
\end{equation}

Similarly, for every equivariant locally free sheaf $\cE$ on $X$, we put
$$s_y(\cE):=\sum_{i\geq 0}(-1)^i[{\rm Sym}^i(\cE)]y^i\in K^0_G\llparenthesis y\rrparenthesis.$$
Again, it is easy to see that given a short exact sequence (\ref{eq_exact_seq}), we have
\begin{equation}\label{eq_add_s}
s_y(\cE)=s_y(\cE_1)\cdot s_y(\cE_2).
\end{equation}

\begin{lem}\label{lem_inverse_lambda}
For a $G$-variety $X$ and an equivariant, locally free sheaf $\cE$ of rank $r$ on $X$, the following hold:
\begin{enumerate}
\item[i)] The inverse of $\lambda_y(\cE)$ in $K_G^0(X)\llbracket y\rrbracket$ is given by $s_y(\cE)$.
\item[ii)] The inverse of $\lambda_{y^{-1}}(\cE)$ in $K_G^0(X)\llparenthesis y\rrparenthesis$ is given by 
$$[{\rm det}(\cE)^{-1}]y^{r}s_{y}(\cE^{\vee}).$$
\end{enumerate}
\end{lem}

\begin{proof}
Note first that the assertion in i) holds if $\cE=\cL$ is a line bundle: in this case 
$\lambda_y(\cL)=1+[\cL]y$,
whose inverse is 
$$\sum_{i\geq 0}(-1)^i[\cL^i]y^i=s_y(\cL).$$
We then deduce using (\ref{eq_add_lambda}) and (\ref{eq_add_s}) that i) holds if $\cE$
admits a filtration whose successive quotients are equivariant line bundles. The general case follows since 
there is a morphism of $G$-varieties $f\colon \widetilde{Z}\to Z$ such that the induced ring homomorphism 
$f^*\colon K^0(Z)\to K^0(\widetilde{Z})$ is injective and such that $f^*(\cE)$ has a filtration 
whose successive quotients are equivariant line bundles (for example, $f$ can be taken a composition 
of suitable projective bundles). 

Note now that if $\cE$ is a locally free sheaf of rank $r$, then for every $i$, with $0\leq i\leq r$, we have
$$\wedge^{r-i}\cE\simeq \wedge^i\cE^{\vee}\otimes_{\cO_X}{\rm det}(\cE).$$
Therefore we can write
$$\lambda_{y^{-1}}(\cE)=y^{-r}\cdot \sum_{i=0}^r[\wedge^i\cE]y^{r-i}=
y^{-r}\cdot \sum_{i=0}^r[\wedge^{r-i}\cE]y^{i}=[{\rm det}(\cE)]y^{-r}\cdot \lambda_{y}(\cE^{\vee}).$$
By i), the inverse of $\lambda_{y^{-1}}(\cE)$ in $K_G^0(X)\llparenthesis y\rrparenthesis$ is then given by 
$[{\rm det}(\cE)^{-1}]y^{r}s_y(\cE^{\vee})$.
\end{proof}

\section{The motivic Chern class}

In this section we review the definition of the motivic Chern class and its description via Saito's theory
of mixed Hodge modules. For the sake of exposition, we first discuss the case when we do not have a group action,
following \cite{BSY}, and then explain the equivariant version.


For a variety $X$, the \emph{motivic Chern class} of Brasselet, Sch\"{u}rmann, and Yokura
is a group homomorphism
$$mC_y\colon K_0({\rm Var}/X)\to K_0(X)[y]$$
(this homomorphism was denoted $mC_*$ in \cite{BSY}, but we prefer
the notation $mC_y$ since it is easy to adapt to a change of variable).
Recall that if $f\colon X'\to X$ is a proper morphism, then we have an induced morphism
of Abelian groups
$f_*\colon K_0(X')\to K_0(X)$ and we extend this to a morphism of Abelian groups $f_*\colon K_0(X')[y]\to K_0(X)[y]$ that maps
$\sum_i\alpha_iy^i$ to  $\sum_if_*(\alpha_i)y^i$. 

With this notation, the motivic Chern class is uniquely characterized by the following two properties:
\begin{enumerate}
\item[1)] If $X$ is smooth and irreducible, of dimension $n$, then\footnote{We make a slight abuse of notation, by denoting
$\lambda_y(\Omega_X)\in K^0(X)[y]$ and its image in $K_0(X)[y]$ via the canonical map $K^0(X)[y]\to K_0(X)[y]$ in the same way.}
$$mC_y(\big[{\rm id}_X]\big)=\lambda_y(\Omega_X).$$
\item[2)] Functoriality with respect to push-forward via proper morphisms:
if $f\colon X'\to X$ is a proper morphism of algebraic varieties, then we have a commutative diagram
$$
\xymatrix{
K_0({\rm Var}/X') \ar[r]^{f_!} \ar[d]_{mC_y} & K_0({\rm Var}/X)
\ar[d]^{mC_y} \\
K_0(X')[y]\ar[r]^{f_*} & K_0(X)[y].
}
$$
\end{enumerate}

It is easy to see, using resolution of singularities, that for every $X$, the Abelian group $K_0({\rm Var}/X)$ 
is generated by $[Y\to X]$, with $Y$ a smooth, irreducible variety, and the morphism $Y\to X$ proper. This immediately 
implies that there is at most one transformation $mC_y$ that satisfies 1) and 2) above. Existence can be proved using
Bittner's presentation \cite{Bittner} in terms of blow-up relations. However, for us it will be useful to have (at least, when $X$ is smooth)
 the explicit description
of the motivic Chern class via the Grothendieck group of mixed Hodge modules, which we now recall. 

Saito constructed in \cite{Saito-MHM}, for every complex algebraic variety $X$, an Abelian category ${\rm MHM}(X)$ consisting
of \emph{mixed Hodge modules}. The corresponding bounded derived category $D^b\big({\rm MHM}(X)\big)$ enjoys a six-functor formalism.
If $X$ is smooth, then a mixed Hodge module consists of a $\cD_X$-module\footnote{In what follows, we always consider 
\emph{left} $\cD_X$-modules.} $\cM$, where $\cD_X$ is the sheaf of differential operators on $X$,
together with a good filtration $F_{\bullet}$ on $\cM$ (the \emph{Hodge filtration}), compatible with the filtration on $\cD_X$ by order of differential operators;
moreover, there is also a rational structure and a weight filtration, but these will not play a role in what follows. We will say that $(\cM,F_{\bullet}\cM)$
\emph{underlies} the corresponding mixed Hodge module. This data is furthermore supposed to satisfy a complex set of conditions.

Every algebraic variety $X$ carries an object $\Q_X^H\in D^b\big({\rm MHM}(X)\big)$. If $X$ is smooth, irreducible, of dimension $n$, then
the shift $\Q_X^H[n]$ is a mixed (in fact, pure) Hodge module, with underlying filtered $\cD_X$-module $(\cO_X,F_{\bullet}\cO_X)$, such that
${\rm Gr}_p^F\cO_X=0$ for $p\neq 0$. 

Recall that if $M$ is a mixed Hodge module on the smooth, irreducible variety $X$, of dimension $n$,
with underlying filtered $\cD_X$-module $(\cM,F_{\bullet}\cM)$, then the 
de Rham complex ${\rm DR}_X(M)$ is the filtered complex
\begin{equation}\label{eq1DeRham}
0\to \cM\to\Omega_X^1\otimes_{\cO_X}\cM\to\ldots\to\Omega_X^n\otimes_{\cO_X}\cM\to 0,
\end{equation}
placed in cohomological degrees $-n,\ldots,0$. The $p^{\rm th}$ graded piece ${\rm Gr}^F_p{\rm DR}_X(M)$ is the
complex of coherent $\cO_X$-modules
\begin{equation}\label{eq2DeRham}
0\to {\rm Gr}^F_p(\cM)\to\Omega_X^1\otimes_{\cO_X}{\rm Gr}_{p+1}^F(\cM)\to\ldots\to\Omega_X^n\otimes_{\cO_X}{\rm Gr}_{p+n}^F(\cM)\to 0.
\end{equation}
Note that if $M=\Q_X^H[n]$, then ${\rm DR}_X(M)$ is the usual algebraic De Rham complex of $X$ (but placed in degrees $-n,\ldots,0$) and
\begin{equation}\label{eq_DR_OX}
{\rm Gr}_{-p}^F{\rm DR}_X\big(\Q_X^H[n]\big)=\Omega_X^p[n-p].
\end{equation}
The above definition extends to give exact functors
$${\rm Gr}^F_p\colon D^b\big({\rm MHM}(X)\big)\to D^b_{\rm coh}(X),$$
where on the right-hand side we have the bounded derived category of coherent sheaves on $X$.

For every variety $X$, we consider the Grothendieck group $K_0\big({\rm MHM}(X)\big)$ of the Abelian category ${\rm MHM}(X)$. This is canonically
identified with the Grothendieck group $K_0\big(D^b({\rm MHM}(X))\big)$ of the corresponding derived category, with the isomorphism
$$K_0\big({\rm MHM}(X)\big)\to K_0\big(D^b({\rm MHM}(X))\big)$$
induced by the embedding ${\rm MHM}(X)\hookrightarrow D^b\big({\rm MHM}(X)\big)$, and whose inverse maps 
$[C^{\bullet}]$ to $\sum_{i\in\Z}(-1)^i[{\mathcal H}^i(C^{\bullet})]$. From now on, we tacitly use this identification.

\begin{rmk}
Note that if $[-]$ is the usual translation functor in a derived category, then it follows from the formula for the above identification that for every
$C^{\bullet}\in D^b\big({\rm MHM}(X)\big)$, if $\alpha$ is the class of $C^{\bullet}$ in $K^0\big({\rm MHM}(X)\big)$, then the class of 
$C^{\bullet}[m]$ is $(-1)^m\alpha$. 
\end{rmk}

For a proper morphism $f\colon X\to Y$, the functor $f_*=f_!\colon D^b\big({\rm MHM}(X)\big)\to D^b\big({\rm MHM}(Y)\big)$
induces a morphism of Abelian groups 
$$f_*\colon K_0\big({\rm MHM}(X)\big)\to K_0\big({\rm MHM}(Y)\big).$$
If $g\colon Y\to Z$ is another proper morphism, then we have $(g\circ f)_*=g_*\circ f_*\colon K_0\big({\rm MHM}(X)\big)\to K_0\big({\rm MHM}(Z)\big)$. 

For every complex algebraic variety $X$, we have a morphism of Abelian groups
$$\chi_{\rm Hdg}\colon K_0({\rm Var}/X)\to K_0\big({\rm MHM}(X)\big)$$
that maps the class of $f\colon Y\to X$ to the class of $f_!(\Q^H_Y)$. For the fact that this is well-defined, see \cite[Section~4]{BSY}.
It is clear that this is compatible with push-forward via proper morphisms: if $f\colon X'\to X$ is such a  morphism, then 
we have a commutative diagram
$$
\xymatrix{
K_0({\rm Var}/X') \ar[r]^{f_!} \ar[d]_{\chi_{\rm Hdg}} & K_0({\rm Var}/X)
\ar[d]^{\chi_{\rm Hdg}} \\
K_0\big({\rm MHM}(X')\big)\ar[r]^{f_*=f_!} & K_0\big({\rm MHM}(X)\big).
}
$$

Suppose now that $X$ is a smooth irreducible variety. We have a morphism of Abelian groups
$$mH_y\colon K_0\big({\rm MHM}(X)\big)\to K_0(X)[y,y^{-1}]$$
which to the class of a mixed Hodge module $M$ associates 
$$\sum_{p\in\Z}\big[{\rm Gr}^F_{-p}{\rm DR}_X(M)\big](-y)^p=\sum_{p,i\in\Z}(-1)^i\big[\cH^i {\rm Gr}^F_{-p}{\rm DR}_X(M)\big](-y)^p.$$
This makes sense since ${\rm Gr}^F_{-p}{\rm DR}_X(M)$ is exact for all but finitely many $p$ and it is well-defined since morphisms of
mixed Hodge modules are strict with respect to the Hodge (and also the weight) filtration. For details, see \cite[Section~4]{BSY}.
It is a consequence of Saito's result concerning the behavior of the de Rham complex with respect to proper push-forwards that this
transformation
is compatible with proper push-forward; more precisely, if
$f\colon X'\to X$ is a proper morphism of smooth varieties, then we have a commutative diagram
$$
\xymatrix{
K_0\big({\rm MHM}(X')\big) \ar[r]^{f_*} \ar[d]_{mH_y} & K_0\big({\rm MHM}(X)\big)
\ar[d]^{mH_y} \\
K_0(X')[y,y^{-1}]\ar[r]^{f_*} & K_0(X)[y,y^{-1}].
}
$$
It is then easy to see that we have
\begin{equation}\label{eq_formula_mC}
mC_y=mH_y\circ\chi_{\rm Hdg}\colon K_0({\rm Var}/X)\to K_0(X)[y,y^{-1}].
\end{equation}
Indeed, since both transformations commute with push-forward by proper maps and since $K_0({\rm Var}/X)$ is generated by classes of
proper morphisms $Y\to X$, with $Y$ smooth and irreducible, it is enough to check that if $Y$ is a smooth, irreducible variety, of dimension $n$,
then 
$$mH_y(\Q_Y^H)=\sum_{p=0}^n[\Omega_Y^p]y^p.$$
This follows from the definition of $mH_y$ and formula (\ref{eq_DR_OX}):
$$mH_y(\Q_Y^H)=(-1)^n\cdot mH_y\big(\Q_Y^H[n]\big)=(-1)^n\cdot\sum_{p\in\Z}[\Omega_Y^p]\cdot (-1)^{n-p}(-y)^p=\sum_{p=0}^n[\Omega_Y^p]y^p.$$

\begin{rmk}\label{rmk_Tate_twist}
Recall that on mixed Hodge modules we have the operation of Tate twisting. Given a mixed Hodge module $M$ with underlying filtered $\cD_X$-module
$(\cM,F_{\bullet}\cM)$, the mixed Hodge module $M(\ell)$ has underlying filtered $\cD_X$-module $(\cM,F_{\bullet-\ell}\cM)$. This implies that
for every $j$ we have 
$${\rm Gr}^F_{-p}{\rm DR}_X\big(M(\ell)\big)={\rm Gr}^F_{-p-\ell}{\rm DR}_X(M),$$
and thus
\begin{equation}\label{eq_Tate_twist}
mH_y\big(M(\ell)\big)=(-y)^{-\ell}\cdot mH_y(M).
\end{equation}
We also note that for every morphism $f\colon X\to Y$, the induced morphism $f_*\colon D^b\big({\rm MHM}(X)\big)\to D^b\big({\rm MHM}(X)\big)$
commutes with Tate twists. 
\end{rmk}

Finally, we need to recall the behavior of the transformation $mH_y$ with respect to duality. For every complex algebraic variety $X$,
we have the contravariant duality functor
$${\mathbf D}\colon D^b\big(MHM(X)\big)\to D^b\big(MHM(X)\big)$$
and an induced morphism of Abelian groups ${\mathbf D}\colon K_0\big({\rm MHM}(X)\big)\to K_0\big({\rm MHM}(X)\big)$.
The duality functor satisfies the following properties:
we have a natural equivalence of functors
\begin{equation}\label{eq1_D}
{\mathbf D}\circ {\mathbf D}\simeq {\rm Id}\quad \text{on}\quad D^b\big({\rm MHM}(X)\big)
\end{equation}
and if $f\colon X\to Y$ is a morphism of smooth algebraic varieties,
then we have a natural equivalence of functors
\begin{equation}\label{eq2_D}
{\mathbf D}\circ f_!\simeq f_*\circ {\mathbf D}\quad\text{on}\quad D^b\big({\rm MHM}(X)\big).
\end{equation}
Duality exhibits the following behavior with respect to Tate twists:
$${\mathbf D}\big(u(m)\big)\simeq \big({\mathbf D}(u))(-m)\quad\text{for every}\quad u\in D^b\big({\rm MHM}(X)\big), m\in\Z.$$
We thus deduce using the last assertion in Remark~\ref{rmk_Tate_twist} and the equivalences (\ref{eq1_D}) and  (\ref{eq2_D}) that for every morphism $f\colon X\to Y$, the functor
$f_!\colon D^b\big({\rm MHM}(X)\big)\to D^b\big({\rm MHM}(X)\big)$ commutes with Tate twists.

Another key property of the duality functor is the following: if $X$ is smooth, then
for every $p$ and every $M\in D^b\big({\rm MHM}(X)\big)$, there is a canonical isomorphism
$${\rm Gr}_{-p}^F{\rm DR}_X({\mathbf D}M)\simeq \RHom\big({\rm Gr}_p^F{\rm DR}_X(M),\omega_X[n]\big),$$
where $n=\dim(X)$ and $\omega_X=\Omega_X^n$ (see \cite[Section~2.4]{Saito-MHP} and also \cite[Lemma~8.4]{Schnell}). 
Consider the morphism of Abelian groups
$\varphi\colon K_0(X)[y,y^{-1}]\to K_0(X)[y,y^{-1}]$ that maps $[\cF]y^i$ to
$\big[\RHom(\cF,\omega_X[n])\big]y^{-i}$ (which is equal to $(-1)^n[\cF^{\vee}\otimes_{\cO_X}\omega_X]y^{-i}$, where
$\cF^{\vee}$ is the dual of $\cF$, if $\cF$ is locally free). With this notation, we see that
\begin{equation}\label{behavior_duality}
mH_y\circ {\mathbf D}=\varphi\circ mH_y\quad\text{on}\quad K_0\big({\rm MHM}(X)\big)
\end{equation}
(see \cite[Corollary~5.19]{Schurmann} for details).


\bigskip

We next turn to the equivariant setting.
Suppose that $G$ is a linear algebraic group and $X$ is a $G$-variety. In this case one can define an equivariant version 
of the motivic Chern class
$$mC_y^G\colon K^G_0({\rm Var}/X)\to K_0^G(X)[y].$$
This is again characterized by the fact that it commutes with proper push-forward and if $X$ is a smooth irreducible $G$-variety,
then
\begin{equation}\label{eq_equivar_char}
mC^G_y\big([{\rm id}_X]\big)=\lambda_y(\Omega_X)\in K_0^G(X)[y].
\end{equation}
Existence can be again proved via an equivariant version of Bittner's presentation of the Grothendieck
group of varieties over $X$ (see \cite{AMSS} and \cite{FRW}). For our purpose, however, it is more important to have
the description via mixed Hodge modules. 

Dealing with mixed Hodge modules in the equivariant setting is more subtle, but fortunately the details have been worked out by
Achar \cite{Achar}. Suppose that $X$ is a $G$-variety. The definition of an equivariant mixed Hodge module on $X$ 
parallels that of an  equivariant coherent sheaf: it 
is a mixed Hodge module $M$,
together with an isomorphism of mixed Hodge modules $\theta^*(M)\simeq p_2^*(M)$, that satisfies the cocycle condition. 
In this way we obtain the Abelian category ${\rm MHM}_G(X)$ and the corresponding Grothendieck group $K_0\big({\rm MHM}_G(X)\big)$.

The subtlety is that in  the equivariant setting, in order to have a 6-functor formalism, one can't simply consider the derived category of
${\rm MHM}_G(X)$ (this issue also arises when constructing the derived category of equivariant coherent sheaves and it is addressed in \cite{BL}). 
Inspired by the construction in \emph{loc. cit.}, Achar constructs a triangulated category  $D^b_G(X)$, together with a bounded nondegenerate
$t$-structure, whose heart is ${\rm MHM}_G(X)$. In particular, the embedding ${\rm MHM}_G(X)\hookrightarrow D^b_G(X)$ induces
a canonical isomorphism of Grothendieck groups
$$K_0\big({\rm MHM}_G(X)\big)\simeq K_0\big(D^b_G(X)\big).$$
For \emph{smooth} varieties,
the category $D^b_G(X)$ enjoys the same 6-functor formalism as $D^b\big({\rm MHM}(X)\big)$ in the non-equivariant case. 

We now turn to equivariant motivic Chern classes (we are grateful to J.~Sch\"{u}rmann for explaining to us some of the 
issues that arise in this setting).
For every smooth $G$-variety $X$, we have a group homomorphism
$$\chi_{\rm Hdg}^G\colon K_0^G\big({\rm Var}/X)\to K_0\big({\rm MHM}_G(X)\big).$$
In order to define this, we follow the approach in \cite[Chapter~4]{AMSS} and 
note first  that if $K_0^G({\rm Sm}/X)$ is the Grothendieck group of $G$-varieties $Y\to X$ over $X$,
with $Y$ smooth, then the natural morphism 
$$K_0^G({\rm Sm}/X)\to K_0^G({\rm Var}/X)$$
is an isomorphism, with the inverse map taking $[Y\to X]$ to $\sum_{i=1}^r[Y_i\to X]$, where $Y=\bigsqcup_{i=1}^rY_i$ is a disjoint union
of $G$-invariant smooth locally closed subsets (the existence of such a decomposition follows easily 
by induction on dimension using the fact that the smooth locus of $Y$ is $G$-invariant). 
The map $\chi_{\rm Hdg}^G$ then sends the 
the class of $f\colon Y\to X$ in $K^0_G({\rm Sm}/X)$ to the element of $K_0\big({\rm MHM}_G(X)\big)$ corresponding to $f_!\Q^H_Y$;
note that in this case $\Q^H_Y$ is an element of $D^b_G(Y)$. The basic properties of $f_!$ for morphisms between smooth varieties guarantee
that this is well-defined and commutes with proper push-forward. 

Suppose now that $X$ is a smooth $G$-variety and $M$ is an equivariant mixed Hodge module on $X$. In this case
each complex ${\rm Gr}^F_pDR_X(M)$ is a complex of equivariant coherent sheaves on $X$. We thus obtain a group homomorphism
$$mH_y^G\colon K_0\big({\rm MHM}_G(X)\big)\to K^G_0(X)[y,y^{-1}]$$
that maps $[M]$ to 
$$\sum_{p\in\Z}\big[{\rm Gr}_{-p}^FDR_X(M)\big](-y)^p\in K^G_0(X)[y,y^{-1}].$$
In order to show that this commutes with proper push-forward, one could argue as follows. First, using Chow's lemma
and resolution of singularities, we see that it is enough to consider projective morphisms. By suitably factoring the morphism,
we can further see that it is enough to treat separately the case of a closed immersion and that of a projection 
$X\times\P^n\to X$. Each case then can be treated as in the proof of \cite[Theorem~2.4]{PS}.

We thus deduce that for every smooth $G$-variety $X$, we have
$mC_y^G=mH_y^G\circ\chi_{\rm Hdg}^G$ on $K^0_G({\rm Var}/X)$; indeed, this follows from the fact that both sides commute
with proper push-forward and take the same value on $1_X$ for every $X$ as above.
 Finally, we also need a version of (\ref{behavior_duality})
in the equivariant setting. In fact, we will only need the corresponding equality on elements in the image of 
$\chi_{\rm Hdg}^G$ (and on Tate twists of such elements). Checking the equality on $\chi^G_{\rm Hdg}\big([{\rm id}_X]\big)$ is easy, using
the definition. By using the compatibility of Grothendieck duality with proper push-forward in the equivariant setting
(see \cite[Theorem~25.2]{Hashimoto}), we deduce that the equality in (\ref{behavior_duality}) holds on 
elements in the image of 
$\chi_{\rm Hdg}^G\colon K^G_0({\rm Var}/X)\to K_0\big({\rm MHM}_G(X)\big)$,
whenever $X$ is a smooth $G$-variety
(recall that $K^G_0({\rm Var}/X)$ is generated by classes of proper morphisms of $G$-varieties $Y\to X$, with $Y$ smooth).
Of course, then the equality also holds also on Tate twists of elements in the image of $\chi_{\rm Hdg}^G$.

\section{Motivic Chern class and Hodge ideals}

Let $G$ be a linear algebraic group and $X$  a smooth, irreducible $G$-variety, of dimension $n$. Given a reduced, $G$-invariant, effective divisor $D$ in $X$, 
we consider the open immersion $j\colon U\hookrightarrow X$, where $U$ is the complement of the support of $D$. Our goal in this section is to relate 
the motivic Chern class of $j$ to the generating function describing the classes of the Hodge ideals of $D$ in the equivariant 
Grothendieck group $K^G_0(X)$. 

Recall that the push-forward $j_*\Q^H_U[n]$ is a mixed Hodge module, whose underlying $\cD_X$-module is
$$\cO_X(*D)=\bigcup_{p\geq 0}\cO_X(pD).$$
The Hodge filtration $F_{\bullet}\cO_X(*D)$ satisfies $F_p\cO_X(*D)=0$ for $p<0$ and
$$F_p\cO_X(*D)\subseteq \cO_X\big((p+1)D\big)\quad\text{for}\quad p\geq 0.$$
The Hodge ideals $I_p(D)$ are characterized by 
$$F_p\cO_X(*D)=I_p(D)\cdot\cO_X\big((p+1)D\big).$$
For details about this setup, see \cite{MP}. 

Since we assume that $D$ is a $G$-invariant divisor, it follows that in our case $j_*\Q^H_U[n]$ has a natural
structure of equivariant mixed Hodge module. In this case, the Hodge ideals $I_p(D)$ are equivariant sheaves 
and we are interested in their classes in $K^G_0(X)$.

The following result allows us to relate the motivic Chern class of the inclusion $U\hookrightarrow X$ with the Hodge filtration on $\cO_X(*D)$.
We consider the equivariant version of the morphism of Abelian groups discussed in the previous section:
$$\varphi\colon K^G_0(X)[y,y^{-1}]\to K^G_0(X)[y,y^{-1}], \quad [\cF]y^i\to 
\big[\RHom(\cF,\omega_X[n])\big]y^{-i}.$$

\begin{prop}\label{first_relation}
With the above notation, we have
$$mH^G_y\big(j_*\Q^H_U[n]\big)=y^n\cdot\varphi\big(mC^G_y(U\hookrightarrow X)\big).$$
\end{prop}

\begin{proof}
The key point is that $\Q^H_U[n]$ is a \emph{pure} Hodge module of weight $n$. The choice of a polarization thus gives
an isomorphism
$${\mathbf D}\big(\Q^H_U[n]\big)\simeq\Q^H_U[n](n).$$
On the other hand, using (\ref{eq1_D}) and (\ref{eq2_D}), we see that $\cO_X(*D)$ is the filtered $\cD_X$-module underlying
$$j_*\Q_U^H[n]\simeq j_*{\mathbf D}{\mathbf D}\big(\Q_U^H[n]\big)\simeq {\mathbf D}j_!\big(\Q^H_U[n](n)\big).$$
If we apply $mH^G_y$, we deduce using (\ref{behavior_duality}) that
$$mH^G_y\big(j_*\Q_U^H[n]\big)=\varphi\left(mH^G_y((j_!\Q_U^H)[n](n))\right)$$
and we conclude using (\ref{eq_formula_mC}), (\ref{eq_Tate_twist}), and $\chi^G_{\rm Hdg}(U\hookrightarrow X)=j_!\Q^H_U$ that
$$mH^G_y\big(j_*\Q_U^H[n]\big)=(-1)^n\varphi\big((-y)^{-n}\cdot mC^G_y(U\hookrightarrow X)\big)=y^n\cdot\varphi\big(mC^G_y(U\hookrightarrow X)\big).$$
This completes the proof of the proposition.
\end{proof}

It is convenient to also have the following invariant that records directly the classes in $K^G_0(X)$ of the graded pieces in the
Hodge filtration of an equivariant mixed Hodge module. If $X$ is a smooth irreducible variety of dimension $n$ and if $(\cM,F_{\bullet}\cM)$ is the filtered
$\cD_X$-module underlying an equivariant mixed Hodge module $M$, then we put
$$\chi^G_y(M):=\sum_{p\in\Z}\big[{\rm Gr}^F_{p}(\cM)\big]y^p\in K^G_0(X)\llparenthesis y\rrparenthesis.$$
Note that $\chi^G_y(M)$ is a Laurent power series in $y$ since $F_{p}\cM=0$ for $p\ll 0$.

\begin{prop}\label{chi_and_mH}
If $X$ is a smooth irreducible $G$-variety of dimension $n$, then
for every equivariant mixed Hodge module $M$, we have in $K^G_0(X)\llparenthesis y\rrparenthesis$ the equality
$$\chi^G_{-y}(M)=[\omega_X^{-1}](-y)^ns_{y}(T_X)\cdot mH^G_{y^{-1}}(M),$$
where $T_X$ is the tangent sheaf of $X$.
\end{prop}

\begin{proof}
It follows from the definition of $mH^G_y(M)$ that we can write in $K^G_0(X)\llparenthesis y^{-1}\rrparenthesis$
$$mH^G_y(M)=\sum_{p\in\Z}[{\rm Gr}^F_{-p}{\rm DR}_X(M)](-y)^p=\sum_{p\in\Z}\sum_{i=0}^n(-1)^{n+i}[\Omega_X^i]\cdot \big[{\rm Gr}^F_{-p+i}(\cM)\big](-y)^{p}$$
$$=(-1)^n\cdot\left(\sum_{i=0}^n[\Omega_X^i]y^i\right)\cdot\left(\sum_{q\in\Z}\big[{\rm Gr}^F_q(\cM)](-y)^{-q}\right)=(-1)^n\lambda_y(\Omega_X^1)\cdot
\chi^G_{-y^{-1}}(M).$$
After replacing $y$ by $y^{-1}$ we obtain in $K^G_0(X)\llparenthesis y\rrparenthesis$ the equality
$$mH^G_{y^{-1}}(M)=(-1)^n\lambda_{y^{-1}}(\Omega_X^1)\cdot \chi^G_{-y}(M).$$
We multiply both sides by $[\omega_X^{-1}](-y)^{n}s_{y}(T_X)$, and using the assertion in Lemma~\ref{lem_inverse_lambda} ii),
we conclude 
$$[\omega_X^{-1}](-y)^{n}s_{y}(T_X)\cdot mH^G_{y^{-1}}(M)=\chi^G_{-y}(M).$$
This completes the proof of the proposition.
\end{proof}

By combining Propositions~\ref{first_relation} and \ref{chi_and_mH}, we obtain the formula for the generating function for the classes
of Hodge ideals.

\begin{thm}\label{thm_main1}
Let $X$ be a smooth irreducible $n$-dimensional $G$-variety.
If $D$ is a $G$-invariant reduced effective divisor on $X$ and
$U=X\smallsetminus {\rm Supp}(D)$, then we have the following equality in $K^G_0(X)\llparenthesis y\rrparenthesis$
$$\sum_{p\geq 0}[I_p(D)]y^p=(-1)^na[\omega_X^{-1}](1-ay)^{-1}s_{-ay}(T_X)\cdot \varphi\big(mC^G_{-a^{-1}y^{-1}}(U\hookrightarrow X)\big),$$
where $a=[\cO_X(-D)]\in K_G^0(X)$.
\end{thm}

\begin{proof}[Proof of Theorem~\ref{thm_main1}]
Let us write
$$L(y):=\sum_{p\geq 0}[I_p(D)]y^p\in K^G_0(X)\llparenthesis y\rrparenthesis.$$
Note that by combining Propositions~\ref{first_relation} and \ref{chi_and_mH}, we obtain
$$\chi_{-y}^G\big(j_*\Q^H_U[n]\big)=[\omega_X^{-1}](-y)^ns_y(T_X)\cdot mH^G_{y^{-1}}\big(j_*\Q^H_U[n]\big)$$
$$=(-1)^n[\omega_X^{-1}]s_y(T_X)\cdot\varphi\big(mC^G_{y^{-1}}(U\hookrightarrow X)\big).$$
On the other hand, since 
$${\rm Gr}_p\cO_X(*D)=F_p\cO_X(*D)/F_{p-1}\cO_X(*D)$$
and 
$$F_p\cO_X(*D)\simeq I_p(D)\otimes_{\cO_X}\cO_X\big((p+1)D\big),$$
it follows that if we put $b=a^{-1}=[\cO_X(D)]\in K_G^0(X)$, then 
$$\chi_{-y}^G\big(j_*\Q^H_U[n]\big)=\sum_{p\geq 0}b^{p+1}[I_p(D)](-y)^p-\sum_{p\geq 0}b^{p+1}[I_p(D)](-y)^{p+1}$$
$$=bL(-by)+byL(-by)=b(1+y)L(-by).$$
We thus conclude that 
$$L(-by)=(-1)^na[\omega_X^{-1}](1+y)^{-1}s_y(T_X)\cdot\varphi\big(mC^G_{y^{-1}}(U\hookrightarrow X)\big).$$
After replacing $y$ by $-ay$, we obtain the formula in the theorem.
\end{proof}

\section{The ${\mathbf C}^*$-equivariant Grothendieck group of coherent sheaves on an affine space}

In this section we review the well-known description of the Grothendieck group $K_0^T(\A^n)$ of $\C^*$-equivariant sheaves on $\A^n$
via Hilbert functions. On the complex affine space $\A^n$, we consider
the standard action of $T=\C^*$, given by
$$\lambda\cdot (u_1,\ldots,u_n)=(\lambda u_1,\ldots,\lambda u_n).$$
In this case, the category of $T$-equivariant coherent sheaves on $\A^n$ is equivalent to the category 
of finitely generated $\Z$-graded $S$-modules, where $S=\C[x_1,\ldots,x_n]$, with the standard grading. Recall that for a $\Z$-graded $S$-module $M$ and
for $q\in\Z$, the graded $S$-module $M(q)$ has the same underlying $S$-module, but $M(q)_j=M_{q+j}$. 

Given a finitely generated $\Z$-graded module $M=\oplus_{j\in\Z}M_j$, we consider its Hilbert series
$$H_M(t)=\sum_{j\in \Z}\dim_{\C}(M_j)t^j\in\Z\llparenthesis t\rrparenthesis.$$
Note first that if 
$$0\to M'\to M\to M''\to 0$$
is a short exact sequence of finitely generated graded $S$-modules, then 
$H_M(t)=H_{M'}(t)+H_{M''}(t)$. By 
mapping the class $[M]$ of $M$ to $H_M(t)$, we obtain a group
homomorphism
$$K_0^T(\A^n)\to \Z\llparenthesis t\rrparenthesis.$$

\begin{prop}\label{prop_Kgroup}
By mapping $[M]$ to $(1-t)^n\cdot H_M(t)$, we obtain 
a group isomorphism 
$$\tau_n\colon K_0^T(\A^n)\to \Z[t,t^{-1}].$$
Moreover, $K_0^T(\A^n)$ is freely generated by $\big[S(q)\big]$, for $q\in \Z$,
and $\tau_n\big(S(q)\big)=t^{-q}$ for every $q\in\Z$. 
\end{prop}

\begin{proof}
It is a consequence of Hilbert's Syzygy Theorem that $K_0^T(\A^n)$
is generated by $\big[S(q)\big]$, for $q\in\Z$. On the other hand, 
a basic computation gives 

\begin{equation}\label{eq_Hilbert_fcn}
H_{S(q)}(t)=t^{-q}\cdot H_S(t)=\frac{t^{-q}}{(1-t)^n}.
\end{equation}

In particular, we see that the Laurent polynomials $(1-t)^n\cdot H_{S(q)}(t)$, for $q\in\Z$, satisfy no linear relations over $\Z$. 
This implies that the isomorphism classes of the $S(q)$, with $q\in\Z$, give a basis of $K_0^T(\A^n)$.
The first assertion in the proposition is now clear as well.
\end{proof}

\begin{rmk}\label{rmk_description_K}
Consider a linear embedding $j\colon \A^m\hookrightarrow\A^n$. Note that this is $T$-equivariant, hence we have
a push-forward homomorphism $j_*\colon K^T_0(\A^m)\to K^T_0(\A^n)$. It follows from the definition of the
isomorphism in Proposition~\ref{prop_Kgroup} that we have a commutative diagram
$$\xymatrix{K_0^T(\A^m)\ar[d]_{\tau_m}\ar[r]^{j_*} & K_0^T(\A^n) \ar[d]_{\tau_n}\\
\Z[t,t^{-1}]\ar[r]^{(1-t)^{n-m}} & \Z[t, t^{-1}],
}$$
in which the bottom horizontal map is given by multiplication with $(1-t)^{n-m}$.
\end{rmk}

\begin{eg}\label{eg_subspace}
If $H$ is a linear subspace of $\C^n$ of dimension $m$, then via the isomorphism in Proposition~\ref{prop_Kgroup},
$[\cO_H]\in K_0^T(\A^n)$ corresponds to $(1-t)^{n-m}$.
\end{eg}

From now on, we will tacitly use the identification of $K_0^T(\A^n)$ with $\Z[t,t^{-1}]$ provided by 
Proposition~\ref{prop_Kgroup} and the corresponding identification of $K_0^T(\A^n)\llparenthesis y\rrparenthesis$
with $\Z[t,t^{-1}]\llparenthesis y\rrparenthesis$. 

\begin{rmk}\label{ring_isom}
Note that the canonical group homomorphism
$$K^0_T(\A^n)\to K^T_0(\A^n)$$
is an isomorphism (this holds, more generally, on smooth quasi-projective varieties, but in our case it follows easily from 
Proposition~\ref{prop_Kgroup}, since the classes of $S(q)$, for $q\in\Z$, also generate $K^0_T(\A^n)$ by Hilbert's Syzygy Theorem).

In particular, we see that $K^T_0(\A^n)$ has a ring structure. The advantage of using the isomorphism $\tau_n$ in Proposition~\ref{prop_Kgroup},
as opposed to the one given by the Hilbert series, is that $\tau_n$ is a ring isomorphism. In order to check this, it is enough to note that
$$\tau_n\big([S(p)]\cdot [S(q)]\big)=\tau_n\big([S(p+q)]\big)=t^{-p-q}=t^{-p}\cdot t^{-q}=\tau_n\big([S(p)]\big)\cdot \tau_n\big([S(q)]\big).$$
\end{rmk}

\begin{eg}\label{eg_tangent_space}
Let's compute the image of $s_y(T_{\A^n})$ via $\tau_n$. Note that $T_{\A^n}$ is the sheaf associated to $S(1)^{\oplus n}$, hence it follows from 
formula (\ref{eq_add_s}) that
$$s_y(T_{\A^n})=s_y\big(S(1)\big)^n.$$
On the other hand, since $\big[S(1)\big]=t^{-1}$, we have
$$s_y\big(S(1)\big)=\sum_{i\geq 0}(-1)^it^{-i}y^i=\frac{1}{1+t^{-1}y}.$$
We thus conclude that 
$$s_y(T_{\A^n})=\frac{1}{(1+t^{-1}y)^n}.$$
\end{eg}

\begin{eg}\label{eg_mC}
Arguing as in the previous example, we see that
$$\lambda_y(\Omega_{\A^n})=\lambda_y\big(S(-1)^{\oplus n}\big)=\lambda_y\big(S(-1)\big)^n=(1+ty)^n.$$
By property (\ref{eq_equivar_char}) in the characterization of the equivariant motivic Chern class, we thus have
$$mC^T_y\big([{\rm id}_{\A^n}]\big)=\lambda_y(\Omega_{\A^n})=(1+ty)^n.$$
If $\Lambda$ is a linear subspace of $\A^n$ of dimension $m$ and $j\colon \Lambda\hookrightarrow X$
is the inclusion, we deduce using Remark~\ref{rmk_description_K} and the fact that the motivic Chern class 
commutes with proper push-forwards that
$$mC^T_y\big([\Lambda\hookrightarrow \A^n]\big)=j_*\left(mC^T_y\big([{\rm id}_{\Lambda}]\big)\right)=(1-t)^{n-m}(1+ty)^{m}.$$
\end{eg}

\begin{rmk}\label{description_phi}
The involution 
$$\varphi\colon K_0^T(\A^n)[y,y^{-1}]\to K_0^T(\A^n)[y,y^{-1}]$$
that appears in Proposition~\ref{first_relation}, can be very easily described via the isomorphism $\tau_n$.
Recall that in general, if $\cE$ is the class of an equivariant locally free sheaf, then $\varphi$ maps $[\cE]y^m$ to 
$(-1)^n[\cE^{\vee}\otimes_{\cO_{\A^n}}\omega_{\A^n}]y^{-m}$. Since $\omega_{\A^n}$ corresponds to $S(-n)$, this implies that
$\varphi$ maps $\big[S(-q)\big]y^m$ to $(-1)^n\big[S(q-n)\big]y^{-m}$. In other words, $\varphi$ is the additive involution on $\Z[t^{\pm 1},y^{\pm 1}]$
that maps each $P(t,y)$ to $(-1)^nt^n\cdot P(t^{-1},y^{-1})$.
\end{rmk}

\section{The case of hyperplane arrangements}

Let $V$ be an $n$-dimensional complex vector space that we identify with $\A^n$. We consider a
\emph{hyperplane arrangement} $\cA$ in $V$; this is a collection of distinct hyperplanes in $V$.
We always assume that $\cA$ is \emph{central}, that is, all hyperplanes 
pass through the origin. Let $d$ be the number of hyperplanes in $\cA$ and put $D_{\cA}=\sum_{H\in\cA}H$, hence $D_{\cA}$ is a reduced effective divisor.
Moreover, it is clear that $D_{\cA}$ is $T$-invariant with respect to the standard action of $T=\C^*$ on $V$.
We denote by $U_{\cA}$ the complement of the support of $D_{\cA}$ in $V$.

We begin by briefly recalling some basic invariants attached to $\cA$. For details, we refer to \cite{OT}. 
The \emph{intersection lattice} $L(\cA)$ consists of all intersections of hyperplanes in $\cA$, ordered by reverse inclusion.
Note that this has a unique minimal element, namely $V$. 
The \emph{M\"{o}bius function} of $L(\cA)$ is the function $\mu\colon L(\cA)\times L(\cA)\to\Z$ characterized by the following properties:
\begin{enumerate}
\item[i)] $\mu(W,W)=1$ for every $W\in L(\cA)$.
\item[ii)] If $W_1<W_2$, then $\sum_{W_1\leq Z\leq W_2}\mu(W_1,Z)=0$.
\item[iii)] If $W_1\not\leq W_2$, then $\mu(W_1,W_2)=0$.
\end{enumerate}
For every $W\in\cL(\cA)$, we put $\mu(W):=\mu(V,W)\in\Z$.

The \emph{Poincar\'{e} polynomial} of $\cA$ is 
$$\pi(\cA,x)=\sum_{W\in L(\cA)}\mu(W)(-x)^{{\rm codim}(W,V)}\in \Z[x].$$
This is a fundamental invariant of $\cA$ (for example, it is equal to the Poincar\'{e} polynomial of $U_{\cA}$, see \cite[Theorem~5.93]{OT}).
It is sometimes convenient to also consider the \emph{characteristic polynomial} of $\cA$, given by 
\begin{equation}\label{def_char_poly}
\chi(\cA,x):=x^n\cdot \pi(\cA,-x^{-1}).
\end{equation}

We will make use of the following property, known as \emph{Deletion-Restriction}, which allows computing the characteristic
polynomial by induction on the number of hyperplanes in $\cA$. If $H_0$ is a hyperplane in $\cA$, then we consider 
the arrangement $\cA'$ in $V$ consisting of all hyperplanes in $\cA$ different from $H_0$ and the arrangement $\cA''$ in $H_0$
consisting of all distinct $H_0\cap H$, for $H\in\cA'$. The Deletion-Restriction property of the characteristic polynomial says that,
with this notation, we have
\begin{equation}\label{eq_DelRes}
\pi(\cA,x)=\pi(\cA',x)+x\cdot\pi(\cA'',x)
\end{equation}
(see \cite[Theorem~2.56]{OT}).

\begin{eg}\label{eg1_Poincare}
If $V=\C^n$ and $\cA$ is the union of $d\leq n$ hyperplanes in general position (that is, such that the sum of the hyperplanes 
has simple normal crossings), then it is easy to see that 
$$\pi(\cA,x)=(1+x)^d$$
(see \cite[Proposition~2.44]{OT} for the formula for the M\"{o}bius function in this case). 
\end{eg}


Our first goal is to compute the class of $U_{\cA}\hookrightarrow V$ in $K_0^T({\rm Var}/V)$.

\begin{prop}\label{class_hypCompl}
If $\eta_V\in K_0^T({\rm Var}/V)$ is the class of a hyperplane in $V$, then
$$[U_{\cA}\hookrightarrow V]=\pi(\cA,-\eta_V).$$
\end{prop}

Similar results have appeared before in the literature: for example, Aluffi showed in \cite[Theorem~1.1]{Aluffi}
that the class of $U_{\cA}$ in 
$K_0({\rm Var}/\C)$ is equal to $\chi(\cA,{\mathbf L})$, where ${\mathbf L}$ is the class of $\A^1$ in $K_0({\rm Var}/\C)$. Note that this is compatible with 
the formula in the above proposition, since $\eta_V^i\in K_0^T({\rm Var}/V)$ maps to ${\mathbf L}^{n-i}$ in $K_0({\rm Var}/\C)$. 

The proof in \emph{loc. cit.} proceeds by making use of the definition of the characteristic polynomial and general properties of the M\"{o}bius function.
The same argument would work in our setting, but we proceed differently, by making use of the Deletion-Restriction formula. 

\begin{proof}[Proof of Proposition~\ref{class_hypCompl}]
It follows from the definition of the product in $K_0^T({\rm Var}/V)$ that if $W$ is a linear subspace of $V$, of codimension $r$,
then $[W\hookrightarrow V]=\eta_V^r\in K_0^T({\rm Var}/V)$. This immediately implies that if $H_0$ is a hyperplane in $V$ and 
$\iota\colon H_0\hookrightarrow V$ is the inclusion, then 
\begin{equation}\label{eq2_class_hypCompl}
\iota_*(\eta_{H_0}^m)=\eta_V^{m+1}\quad\text{for every}\quad m\geq 0,
\end{equation}
where $\eta_{H_0}\in  K_0^T({\rm Var}/H_0)$ is the class of a hyperplane in $H_0$. 

We prove the assertion in the proposition by induction on the number $d$ of hyperplanes in $\cA$. If $d=1$, then it follows immediately from
the definition of the Poincar\'{e} polynomial that $\pi(\cA,x)=1+x$ and
$$[U_{\cA}\hookrightarrow V]=1-\eta_V=\pi(\cA,-\eta_V).$$

For the induction step, choose a hyperplane $H_0$ in $\cA$ and let $\cA'$ and $\cA''$ be the hyperplane arrangements defined from $\cA$ and $H_0$,
which appear in the Deletion-Restriction formula. 
It is clear that we have
$U_{\cA}\subseteq U_{\cA'}$ and 
$$U_{\cA'}\smallsetminus U_{\cA}=U_{\cA''}.$$
We thus conclude that
$$[U_{\cA}\hookrightarrow V]=[U_{\cA'}\hookrightarrow V]-\iota_*\big([U_{\cA''}\hookrightarrow H_0]\big).$$
Using the induction hypothesis and formula (\ref{eq2_class_hypCompl}), we obtain
$$[U_{\cA}\hookrightarrow V]=\pi(\cA',-\eta_V)-\iota_*\big(\pi({\cA''},-\eta_{H_0})\big)=
\pi(\cA',-\eta_V)-\eta_{V}\cdot\pi(\cA'',-\eta_V)=\pi(\cA,-\eta_V),$$
where the last equality follows from the Deletion-Restriction formula (\ref{eq_DelRes}).
This completes the proof of the proposition.
\end{proof}

In order to give the formula for the equivariant motivic Chern class of $U_{\cA}$, it is more convenient to use
the characteristic polynomial of $\cA$. 

\begin{cor}\label{mC_hyp_arr}
Via the isomorphism $\tau_n$ in Proposition~\ref{prop_Kgroup}, the
equivariant motivic Chern class 
$mC_y^{T}\big([U_{\cA}\hookrightarrow V]\big)\in K_0^T(V)[y]$ corresponds to
$(1-t)^n\cdot \chi\big(\cA, (1+ty)/(1-t)\big)$.
\end{cor}

\begin{proof}
For every $i\geq 0$, we see that $\eta_V^i\in K_0^T({\rm Var}/V)$ is the class $[\Lambda_i\hookrightarrow V]$, where $\Lambda_i$
is a linear subspace of $V$ of codimension $i$. It follows that the equivariant motivic Chern class maps  $\eta_V^i$ to the element in
$K_0^T({\rm Var}/V)$ that corresponds to $(1-t)^i(1+ty)^{n-i}$ via $\tau_n$
(see Example~\ref{eg_mC}). If the characteristic polynomial of $\cA$ is $\chi(\cA,x)=\sum_{i=0}^na_ix^{n-i}$, then
$\pi(\cA,x)=(-1)^n\sum_{i=0}^n(-1)^{n-i}a_ix^i$. By Proposition~\ref{class_hypCompl}, we see that
$$[U_{\cA}\hookrightarrow V]=\sum_{i=0}^na_i\eta_V^i.$$
It follows that the equivariant motivic Chern class maps this to the element of $K_0^T({\rm Var}/V)$ that, via $\tau_n$, corresponds to
$$\sum_{i=0}^na_i(1-t)^i(1+ty)^{n-i}=(1-t)^n\cdot \chi\big(\cA, (1+ty)/(1-t)\big).$$
\end{proof}

\begin{rmk}
After a first version of this article was made public, we became aware of a recent result of Liao, giving a similar formula for the motivic Chern class
of ${\mathbf P}(U_{\cA})\to {\mathbf P}(V)$, where ${\mathbf P}(U_{\cA})$ is the complement  of the corresponding projective hyperplane arrangement in the
projective space ${\mathbf P}(V)$ of lines in $V$ (see \cite[Theorem~5.2]{Liao}). 
\end{rmk}

We can now prove the formula for the generating function of the Hilbert series of 
the Hodge ideals of $D_{\cA}$, stated in the Introduction.

\begin{proof}[Proof of Theorem~\ref{thm_main2}]
We use the formula in Theorem~\ref{thm_main1} and make explicit the terms in that formula via the isomorphism $\tau_n$ in Proposition~\ref{prop_Kgroup}.
Recall that by definition, for every graded $S$-module $M$, we have
$$H_M(t)=\frac{1}{(1-t)^n}\tau_n\big([M]\big).$$
It follows that the formula in the statement of the theorem holds if we show, identifying $K_0^T(V)$ to $\Z[t,t^{-1}]$ via $\tau_n$, that
$$\sum_{p\geq 0}\big[I_p(D_{\cA})]y^p=\frac{t^{d}}{(1-t^dy)}\cdot\pi\big(\cA,(1-t)/t(1-t^{d-1}y)\big).$$

We put $S={\rm Sym}^{\bullet}(V)$, with the standard grading.
Since $\cA$ consists of $d$ hyperplanes, we have
$$a:=\big[\cO_V(-D_{\cA})\big]=\big[S(-d)\big]=t^d.$$
This implies that 
\begin{equation}\label{eq_101}
(1-ay)^{-1}=(1-t^dy)^{-1}.
\end{equation}
Note also that 
\begin{equation}\label{eq_102}
[\omega_V^{-1}]=\big[S(n)\big]=t^{-n}.
\end{equation}

We have seen in Example~\ref{eg_tangent_space} that
$$s_y(T_V)=\frac{1}{(1+t^{-1}y)^n},$$
hence
\begin{equation}\label{eq_103}
s_{-ay}(T_V)=\frac{1}{(1-t^{d-1}y)^n}.
\end{equation}

Finally, we need to compute $\varphi\big(mC^T_{-a^{-1}y^{-1}}(U_{\cA}\hookrightarrow V)\big)$. Note first that
Corollary~\ref{mC_hyp_arr} gives
$$mC_y^T(U_{\cA}\hookrightarrow V)=(1-t)^n\chi\big(\cA, (1+ty)/(1-t)\big),$$
and thus
$$mC^T_{-a^{-1}y^{-1}}(U_{\cA}\hookrightarrow V)=(1-t)^n\chi\big(\cA,(1-t^{-(d-1)}y^{-1})/(1-t)\big).$$
Using the description of $\varphi$ in Remark~\ref{description_phi}, we obtain
$$\varphi\big(mC^T_{-a^{-1}y^{-1}}(U_{\cA}\hookrightarrow V)\big)
=(-1)^nt^n(1-t^{-1})^n\cdot \chi\big(\cA, (1-t^{d-1}y)/(1-t^{-1})\big)$$
$$=(1-t)^n\cdot
\chi\big(\cA, t(1-t^{d-1}y)/(t-1)\big)=(-1)^nt^n(1-t^{d-1}y)^n\cdot\pi\big(\cA,(1-t)/t(1-t^{d-1}y)\big),$$
where the last equality follows from (\ref{def_char_poly}).
We thus conclude using (\ref{eq_101}), (\ref{eq_102}), (\ref{eq_103}), and Theorem~\ref{thm_main1} 
that
$$\sum_{p\geq 0}\big[I_p(D_{\cA})\big]y^p=\frac{t^{d}}{(1-t^dy)}\cdot\pi\big(\cA,(1-t)/t(1-t^{d-1}y)\big).$$
This completes the proof of the theorem.
\end{proof}

\begin{proof}[Proof of Corollary~\ref{cor_main2}]
Recall that the first Hodge ideal $I_0(D_{\cA})$ can be identified with the multiplier ideal $\cI\big((1-\epsilon)D_{\cA}\big)$, for $0<\epsilon\ll 1$ (see \cite[Proposition~10.1]{MP}). 
By making $y=0$ in Theorem~\ref{thm_main2}, we obtain the assertion in the corollary.
\end{proof}

\section{An example: simple normal crossing arrangements}

If $\cA$ consists of $d\leq n$ linear hyperplanes in $V=\C^n$ in general position, then it follows from
Example~\ref{eg1_Poincare} and Theorem~\ref{thm_main2} that
\begin{equation}\label{eq_gen_position}
\sum_{k\geq 0}H_{I_k(D_{\cA})}(t)y^k=\frac{(1-t^dy)^{d-1}}{(1-t)^n(1-t^{d-1}y)^d}.
\end{equation}

In this section we give a direct computation of the Hilbert functions of the ideals $I_k(D_{\cA})$ in this case,
using the explicit description of these ideals, and recover the formula in  (\ref{eq_gen_position}). 
After a suitable linear change of coordinates, we may assume that $D_{\cA}$ is the divisor defined by $f=\prod_{i=1}^dx_i$.

Recall that in this case, the Hodge filtration is given by 
$$F_p\cO_V(*D_{\cA})=F_p\cD_V\cdot \cO_V(D_{\cA}).$$
This implies that $F_p\cO_V(*D_{\cA})$ is generated over $S=\C[x_1,\ldots,x_n]$ by the Laurent monomials
$x_1^{a_1}\cdots x_d^{a_d}$, with $a_i\leq -1$ for $1\leq i\leq d$ and $\sum_{i=1}^da_i\geq-(d+p)$
(see \cite[Section~8]{MP}). It is then easy to see that $F_p\cO_V(*D_{\cA})$ has a basis over $\C$ given by
the Laureant monomials $x_1^{b_1}\cdots x_n^{b_n}$ that satisfy
\begin{equation}\label{eq1_example_SNC}
\quad\quad\sum_{i=1}^d\min\{b_i,-1\}\geq -(d+p)\quad\quad\text{and}\quad b_i\geq 0\quad\text{for}\quad i>d.
\end{equation}

Since $F_p\cO_V(*D_{\cA})$ is $\Z^n$-graded, with respect to the standard $\Z^n$-grading on the ring $\C[x_1,\ldots,x_n]$,
it is convenient to first compute 
$$H_p(t_1,\ldots,t_n):=\sum_{(b_1,\ldots,b_n)}t_1^{b_1}\cdots t_n^{b_n},$$
where $(b_1,\ldots,b_n)$ runs over the tuples that satisfy (\ref{eq1_example_SNC}). It is then clear that $H_p(t,\ldots,t)$ is the Hilbert series
of $F_p\cO_V(*D_{\cA})$. 

We write
$H_p=\sum_{J}H_p^J$, where $J$ varies over all subsets of $\{1,\ldots,d\}$, and where $H_p^J$ is the sum of all monomials in $H_p$ that correspond to those 
$(b_1,\ldots,b_n)$ such that $b_i<0$ if and only if $i\in J$. Note that if $(b_1,\ldots,b_n)$ satisfies this condition, then 
it also satisfies (\ref{eq1_example_SNC}) if and only if $\sum_{i\in J}b_i\geq -p-|J|$. 
We thus obtain
$$H_p(t_1,\ldots,t_n)=\sum_{q=0}^d\sum_{|J|=q}\left(\prod_{i\not\in J}\frac{1}{1-t_i}\cdot\prod_{\sum_{i\in J}b_i\geq -p-q}\prod_{i\in J}t_i^{b_i}\right).$$
Given $J$ with $|J|=q$ and
 $(b_i)_{i\in J}$, if we write $b_i=-1-\gamma_i$, the conditions $b_i<0$ for all $i\in J$ and $\sum_{i\in J}b_i\geq -p-q$ are equivalent to
$\gamma_i\geq 0$ for all $i\in J$ and $\sum_{i\in J}\gamma_i\leq p$. Note that 
we have ${{m+q-1}\choose m}$ such tuples with $\sum_{i\in J}\gamma_i=m\leq p$
and for every such $(\gamma_i)_{i\in J}$, we have $\sum_{i\in J}b_i=-q-m$.
Since we have ${d\choose q}$ subsets of $\{1,\ldots,d\}$ with $q$ elements, 
we obtain
$$H_{F_p\cO_V(*D_{\cA})}(t)=H_p(t,\ldots,t)=\sum_{q=0}^d{d\choose q}\frac{1}{(1-t)^{n-q}}\cdot\sum_{m=0}^p{{m+q-1}\choose m}t^{-q-m}.$$

We can now compute the generating function
$$\sum_{p\geq 0}H_{F_p\cO_V(*D_{\cA})}(t)y^p=\sum_{p\geq 0}\sum_{q=0}^d{d\choose q}\frac{1}{(1-t)^{n-q}}\cdot\sum_{m=0}^p{{m+q-1}\choose m}t^{-q-m}
y^p.$$
In order to do this, let us write
$$\sum_{p\geq 0}\sum_{m=0}^p{{m+q-1}\choose m}t^{-q-m}y^p=\sum_{m\geq 0}\sum_{p\geq m}{{m+q-1}\choose m}t^{-q-m}y^p$$
$$=\frac{1}{(1-y)}\cdot\sum_{m\geq 0}{{m+q-1}\choose m}t^{-q-m}y^m=\frac{t^{-q}}{(1-y)(1-t^{-1}y)^q}.$$
We thus see that 
$$\sum_{p\geq 0}H_{F_p\cO_V(*D_{\cA})}(t)y^p=\sum_{q=0}^d{d\choose q}\frac{t^{-q}}{(1-t)^{n-q}(1-y)(1-t^{-1}y)^q}$$
$$=\frac{1}{(1-t)^{n-d}(1-y)}\left(\frac{t^{-1}}{1-t^{-1}y}+\frac{1}{(1-t)}\right)^d
=\frac{t^{-d}(1-y)^{d-1}}{(1-t)^n(1-t^{-1}y)^d}.$$
Since 
$$F_p\cO_V(*D_{\cA})\simeq I_p(D_{\cA})\otimes_{\cO_V}\cO_V\big((p+1)D_{\cA}\big)\simeq I_p(D_{\cA})\otimes_{\cO_V}\cO_V\big(d(p+1)\big),$$
it follows that
$$H_{I_p(D_{\cA})}(t)=t^{d(p+1)}\cdot H_{F_p\cO_V(*D_{\cA})}(t).$$
We conclude that
$$\sum_{p\geq 0}H_{I_p(D_{\cA})}(t)y^p=t^d\cdot\sum_{p\geq 0}H_{F_p\cO_V(*D_{\cA})}(t)(t^dy)^p
=\frac{(1-t^dy)^{d-1}}{(1-t)^n(1-t^{d-1}y)^d}.$$
We thus recover the formula in (\ref{eq_gen_position}).


\section*{References}
\begin{biblist}

\bib{Achar}{article}{
author={Achar, P.},
title={Equivariant mixed Hodge modules},
journal={preprint, Lecture notes from the Clay Mathematics Institute workshop on Mixed Hodge Modules and Applications},
pages={available at \textit{https://www.math.lsu.edu/}${\tilde{}}$\textit{pramod/docs/emhm.pdf}},
date={2013},
}

\bib{Aluffi}{article}{
   author={Aluffi, P.},
   title={Grothendieck classes and Chern classes of hyperplane arrangements},
   journal={Int. Math. Res. Not. IMRN},
   date={2013},
   number={8},
   pages={1873--1900},
}


\bib{AMSS}{article}{
author={Aluffi, P.},
author={Mihalcea, L.},
author={Sch\"{u}rmann, J.},
author={Su, C.},
title={Motivic Chern classes of Schubert cells, Hecke algebras, and applications to Casselman's problem},
journal={preprint arXiv:1902.10101},
date={2019},
}



\bib{BL}{book}{
   author={Bernstein, J.},
   author={Lunts, V.},
   title={Equivariant sheaves and functors},
   series={Lecture Notes in Mathematics},
   volume={1578},
   publisher={Springer-Verlag, Berlin},
   date={1994},
   pages={iv+139},
}

\bib{Bittner}{article}{
   author={Bittner, F.},
   title={The universal Euler characteristic for varieties of characteristic
   zero},
   journal={Compos. Math.},
   volume={140},
   date={2004},
   number={4},
   pages={1011--1032},
}



\bib{BSY}{article}{
   author={Brasselet, J.-P.},
   author={Sch\"{u}rmann, J.},
   author={Yokura, S.},
   title={Hirzebruch classes and motivic Chern classes for singular spaces},
   journal={J. Topol. Anal.},
   volume={2},
   date={2010},
   number={1},
   pages={1--55},
}

\bib{FRW}{article}{
author={Feh\'{e}r, L. M.},
author={Rim\'{a}nyi, R.},
author={Weber, A.},
title= {Motivic Chern classes and K-theoretic stable envelopes},
journal={preprint arXiv:1802.01503},
date={2018},
}


\bib{Hashimoto}{article}{
   author={Hashimoto, M.},
   title={Equivariant twisted inverses},
   conference={
      title={Foundations of Grothendieck duality for diagrams of schemes},
   },
   book={
      series={Lecture Notes in Math.},
      volume={1960},
      publisher={Springer, Berlin},
   },
   date={2009},
   pages={261--478},
}

\bib{Liao}{article}{
   author={Liao, X.},
   title={K-theoretic defect in Chern class identity for a free divisor},
   journal={Int. Math. Res. Not. IMRN},
   date={2019},
   number={19},
   pages={6113--6135},
}
	

\bib{Mustata}{article}{
   author={Musta\c{t}\u{a}, M.},
   title={Multiplier ideals of hyperplane arrangements},
   journal={Trans. Amer. Math. Soc.},
   volume={358},
   date={2006},
   number={11},
   pages={5015--5023},
}

\bib{MP0}{article}{
   author={Musta\c{t}\u{a}, M.},
   author={Popa, M.},
   title={Hodge ideals for ${\bf Q}$-divisors: birational approach},
   journal={J. \'{E}c. polytech. Math.},
   volume={6},
   date={2019},
   pages={283--328},
  }

\bib{MP}{article}{
   author={Musta\c{t}\u{a}, M.},
   author={Popa, M.},
   title={Hodge ideals},
   journal={Mem. Amer. Math. Soc.},
   volume={262},
   date={2019},
   number={1268},
   pages={v+80},
}

\bib{OT}{book}{
   author={Orlik, P.},
   author={Terao, H.},
   title={Arrangements of hyperplanes},
   series={Grundlehren der Mathematischen Wissenschaften [Fundamental
   Principles of Mathematical Sciences]},
   volume={300},
   publisher={Springer-Verlag, Berlin},
   date={1992},
   pages={xviii+325},
}


\bib{PS}{article}{
   author={Popa, Mihnea},
   author={Schnell, Christian},
   title={Generic vanishing theory via mixed Hodge modules},
   journal={Forum Math. Sigma},
   volume={1},
   date={2013},
   pages={e1, 60},
}


\bib{Saito-MHP}{article}{
   author={Saito, M.},
   title={Modules de Hodge polarisables},
   journal={Publ. Res. Inst. Math. Sci.},
   volume={24},
   date={1988},
   number={6},
   pages={849--995},
}


\bib{Saito-MHM}{article}{
   author={Saito, M.},
   title={Mixed Hodge modules},
   journal={Publ. Res. Inst. Math. Sci.},
   volume={26},
   date={1990},
   number={2},
   pages={221--333},
}



\bib{Schnell}{article}{
   author={Schnell, C.},
   title={On Saito's vanishing theorem},
   journal={Math. Res. Lett.},
   volume={23},
   date={2016},
   number={2},
   pages={499--527},
  }


\bib{Schurmann}{article}{
   author={Sch\"{u}rmann, J.},
   title={Characteristic classes of mixed Hodge modules},
   conference={
      title={Topology of stratified spaces},
   },
   book={
      series={Math. Sci. Res. Inst. Publ.},
      volume={58},
      publisher={Cambridge Univ. Press, Cambridge},
   },
   date={2011},
   pages={419--470},
}

\bib{Teitler}{article}{
   author={Teitler, Z.},
   title={A note on Musta\c{t}\u{a}'s computation of multiplier ideals of hyperplane
   arrangements},
   journal={Proc. Amer. Math. Soc.},
   volume={136},
   date={2008},
   number={5},
   pages={1575--1579},
}

\bib{Thomason}{article}{
   author={Thomason, R. W.},
   title={Algebraic $K$-theory of group scheme actions},
   conference={
      title={Algebraic topology and algebraic $K$-theory},
      address={Princeton, N.J.},
      date={1983},
   },
   book={
      series={Ann. of Math. Stud.},
      volume={113},
      publisher={Princeton Univ. Press, Princeton, NJ},
   },
   date={1987},
   pages={539--563},
}



\end{biblist}
\end{document}